\pdfoutput=1
\documentclass[12pt]{article}
\usepackage{graphicx}
\usepackage{amsfonts}
\usepackage{subfig}
\usepackage{amsmath}
\usepackage{soul}
\usepackage{amsthm}
\usepackage[T1]{fontenc}
\usepackage[utf8]{inputenc}
\usepackage{authblk}
\usepackage[utf8]{inputenc}
\usepackage[english]{babel}
\usepackage{xcolor}

\usepackage[round]{natbib}
\usepackage[colorlinks,citecolor=black,urlcolor=blue,linkcolor=blue]{hyperref}
\usepackage{tikz}

\allowdisplaybreaks
\usepackage[margin=1in]{geometry}

\newtheorem{theorem}{Theorem}
\newtheorem{lemma}{Lemma}

\newtheorem{definition}{Definition}
\newtheorem{condition}{Condition}
\title{Strong uniform laws of large numbers for bootstrap means and other randomly weighted sums}
\author{Neil A. Spencer and Jeffrey W. Miller}
\date{}

\begin{document}

\maketitle

\begin{abstract}
This article establishes novel strong uniform laws of large numbers for randomly weighted sums such as bootstrap means. 
By leveraging recent advances, these results extend previous work in their general applicability to a wide range of weighting procedures and in their flexibility with respect to the effective bootstrap sample size.
In addition to the standard multinomial bootstrap and the $m$-out-of-$n$ bootstrap, our results apply to a large class of randomly weighted sums involving negatively orthant dependent (NOD) weights, including the Bayesian bootstrap, jackknife, resampling without replacement, simple random sampling with over-replacement, independent weights, and multivariate Gaussian weighting schemes.
Weights are permitted to be non-identically distributed and possibly even negative.
Our proof technique is based on extending a proof of the i.i.d.\ strong uniform law of large numbers to employ strong laws for randomly weighted sums; in particular, we exploit a recent Marcinkiewicz--Zygmund strong law for NOD weighted sums.
\end{abstract}

\section{Introduction}

The bootstrap \citep{efron1994introduction} and related resampling procedures such as bagging \citep{breiman1996bagging}, the Bayesian bootstrap \citep{rubin1981bayesian}, and the jackknife \citep{efron1982jackknife} are widely used general-purpose tools for statistical inference.
In addition to its original purpose of approximating sampling distributions of estimators \citep{efron1979bootstrap}, the bootstrap and its relatives have been applied to a variety of statistical tasks, including model averaging \citep{breiman1996bagging}, approximate Bayesian inference \citep{newton1994approximate}, outlier detection \citep{singh2003bootlier}, robust Bayesian inference \citep{huggins2019robust,10.1214/21-BA1301}, and causal inference \citep{little2019causal}.  

Due to its versatility, extensions of the bootstrap are frequently proposed to address new statistical questions. When establishing the properties of such methods, bootstrap versions of classical asymptotic results play a key role, such as the weak law of large numbers \cite[Theorem 1]{athreya1984laws}, the strong law of large numbers \cite[Theorem 2]{athreya1984laws}, and the central limit theorem \citep{singh1981asymptotic} for bootstrap means.

Meanwhile, it is sometimes important to obtain convergence over an entire collection of random variables simultaneously, thus guaranteeing convergence for even the worst case in the collection. 
To this end, several authors have established uniform laws of large numbers for bootstrap means.  \citet[Theorems 2.6 and 3.5]{gine1990bootstrapping} proved weak uniform laws of large numbers for the standard multinomial bootstrap, that is, with $\mathrm{Multinomial}\big(n,  (1/n, \ldots, 1/n))$ weights.
\citet[Theorem 3.6.16]{vaart1996weak} proved an analogous result for exchangeably weighted sums such as the Bayesian bootstrap.  As weak laws, these show convergence in probability, but often one needs almost sure convergence, that is, a strong law.





Strong uniform laws of large numbers for bootstrap means are provided by \citet[Section 10.2]{kosorok2008introduction} for weighted sums with independent and identically distributed (i.i.d.)\ weights, with weights obtained by normalizing $n$ i.i.d.\ random variables, and with $\mathrm{Multinomial}\big(n,  (1/n, \ldots, 1/n))$ weights (Theorem 10.13, Corollary 10.14, and Theorem 10.15, respectively). However, these results do not apply to more general schemes such as the jackknife, resampling without replacement, and $\mathrm{Multinomial}\big(m_n, (p_{n 1},\ldots,p_{n n}))$ weights. 

In this article, we present new strong uniform laws of large numbers for randomly weighted sums, aiming to fill these gaps in the literature.  Our first result applies to the case of $\mathrm{Multinomial}\big(m_n, (1/n,\ldots,1/n))$ weights, known as the $m$-out-of-$n$ bootstrap. Our second and third results apply more generally to a large class of randomly weighted sums that involve negatively orthant dependent (NOD) weights.
This covers a wide range of weighting schemes, including the Bayesian bootstrap \citep{rubin1981bayesian}, various versions of the jackknife \citep{efron1982jackknife,chatterjee2005generalized}, resampling without replacement \citep{bickel2012resampling}, simple random sampling with over-replacement \citep{antal2011simple}, $\mathrm{Multinomial}\big(m_n, (p_{n 1},\ldots,p_{n n}))$ weights \citep[Section 3]{antal2011direct}, independent weights \citep{newton1994approximate}, and even schemes involving negative weights such as multivariate Gaussian weights with non-positive correlations \citep{patak2009generalized}.
All three theorems are flexible in terms of the effective bootstrap sample size $m_n$ (that is, the sum of the weights), for instance, allowing $m_n = o(n)$ which is of particular interest for certain applications \citep{bickel2012resampling,10.1214/21-BA1301}.

The article is organized as follows. We present our main results in Section~\ref{section:main-results}, Section~\label{sec:examples} provides some examples of relevant re-weighting schemes, and the proofs of the main results are provided in Section~\ref{section:proofs}. 

\section{Main results}
\label{section:main-results}

We present three strong uniform laws of large numbers: Theorem~\ref{theorem:ulln1} for the multinomial bootstrap, Theorem~\ref{theorem:ulln2} for more general randomly weighted sums, and Theorem~\ref{theorem:ulln3} which establishes faster convergence rates under stronger regularity and moment conditions than Theorem~\ref{theorem:ulln2}. All three results are obtained via extensions of the proof of the i.i.d.\ strong uniform law of large numbers presented by \cite{GhoshJ.K2003BN}. Specifically, our proof of Theorem~\ref{theorem:ulln1} involves replacing the traditional strong law of large numbers with a strong law of large numbers for bootstrap means as presented by \cite{arenal1996unconditional}. Similarly, Theorems~\ref{theorem:ulln2} and \ref{theorem:ulln3} rely on a strong law of large numbers for randomly weighted sums of random variables presented by \cite{chen2019strong}. 

\begin{condition}
\label{condition:function-class}
Suppose $\Theta$ is a compact subset of a separable metric space. Let $u \in \mathbb{N}$ and $H(\theta,x)$ be a real-valued function on $\Theta \times \mathbb{R}^u$ such that
\begin{itemize}
\item[(i)] for each $x \in \mathbb{R}^u$, $\theta\mapsto H(\theta, x)$ is continuous on $\Theta$, and
\item[(ii)] for each $\theta \in \Theta$, $x\mapsto H(\theta,x)$ is a measurable function on $\mathbb{R}^u$.
\end{itemize}
\end{condition}

\begin{theorem}\label{theorem:ulln1}
Let $X_1,X_2,\ldots \in\mathbb{R}^u$ i.i.d.\ and for $n\in\mathbb{N}$ independently, let
\begin{align*}
(W_{n 1},\ldots,W_{n n}) \sim \mathrm{Multinomial}\big(m_n,  (1/n, \ldots, 1/n)\big)
\end{align*}
independently of $(X_1,X_2,\ldots)$,
where $m_n$ is a positive integer for $n\in\mathbb{N}$. Assume Condition~\ref{condition:function-class} and suppose there exists $\delta \in [0,1)$ such that
\begin{align}
\lim_{n \rightarrow \infty} \frac{n^{1 - \delta} \log(n)}{m_n} = 0 \text{~~ and ~~}\mathbb{E}\left(\sup_{\theta \in \Theta} \left|H(\theta, X_1) \right|^{\frac{1}{1 - \delta}}\right) < \infty \label{assump1},
\end{align}
Then
\begin{align}\label{equation:ulln1}
\sup_{\theta \in \Theta} \bigg|\frac{1}{m_n} \sum_{j=1}^{n} W_{nj} H(\theta, X_j) - \mathbb{E}(H(\theta, X_1)) \bigg| \xrightarrow[n\to\infty]{\mathrm{a.s.}} 0. 
\end{align}

\end{theorem}

In Theorems~\ref{theorem:ulln2} and \ref{theorem:ulln3}, we generalize beyond the standard multinomial bootstrap to weighting schemes involving NOD random variables \citep{lehmann1966some,chen2019strong}.

\begin{definition}
A finite collection of random variables $X_1, \ldots, X_n\in\mathbb{R}$ is said to be \emph{negatively orthant dependent (NOD)} if 
\begin{align}
\mathbb{P}\left(X_1 \leq x_1, \ldots, X_n \leq x_n \right) &\leq \prod_{i=1}^n \mathbb{P}(X_i \leq x_i)
\end{align}
and
\begin{align}
\mathbb{P}\left(X_1 > x_1, \ldots, X_n > x_n \right) &\leq \prod_{i=1}^n \mathbb{P}(X_i > x_i)
\end{align}
for all $x_1, \ldots, x_n \in \mathbb{R}$. An infinite collection of random variables is NOD if every finite subcollection is NOD.
\end{definition}
Any collection of independent random variables is NOD, and many commonly used multivariate distributions are NOD including the multinomial distribution, the Dirichlet distribution, the Dirichlet-multinomial distribution, the multivariate hypergeometric distribution, convolutions of multinomial distributions, and multivariate Gaussian distributions for which the correlations are all non-positive \citep{joag1983negative}.

\begin{theorem}\label{theorem:ulln2}
Let $X_1,X_2,\ldots \in \mathbb{R}^u$ i.i.d.\ and for $n\in\mathbb{N}$ independently, let $W_{n 1},\ldots,W_{n n}\in\mathbb{R}$ be NOD random variables, independent of $(X_1,X_2,\ldots)$.
Assume Condition~\ref{condition:function-class} and suppose $\mathbb{E}(\sup_{\theta\in\Theta}|H(\theta,X_1)|^{\beta}) < \infty$, $\sum_{j=1}^n \mathbb{E}(|W_{n j}|^\alpha) = O(n)$, and $\sum_{j=1}^n \mathbb{E}(|W_{n j}|) = O(n^{1/p})$ where $p \in [1,2)$ and $\alpha > 2p$ and $\beta > 1$  satisfy $\alpha^{-1} + \beta^{-1} = p^{-1}$. Then
\begin{align} 
\sup_{\theta \in \Theta} \bigg|\frac{1}{n^{1/p}} \sum_{j=1}^{n} \Big( W_{nj} H(\theta, X_j) - \mathbb{E}(W_{nj} H(\theta, X_j))\Big) \bigg| \xrightarrow[n\to\infty]{\mathrm{a.s.}} 0.  \label{equation:ulln2}
\end{align}
\end{theorem}
In particular, if $n^{-1/p}\sum_{j=1}^n \mathbb{E}(W_{n j}) \to 1$, then Equation~\ref{equation:ulln2} is analogous to Equation~\ref{equation:ulln1} with $n^{1/p}$ in place of $m_n$.
While the moment condition on $H(\theta,X_1)$ in Theorem~\ref{theorem:ulln2} is slightly more stringent than in Theorem~\ref{theorem:ulln1}, it applies to a more general class of resampling procedures.
For instance, the distribution of $W_{n j}$ can be different for each $n$ and $j$; in particular, there is no assumption that $W_{n 1},\ldots,W_{n n}$ are exchangeable or even identically distributed.
Further, the weights $W_{n j}$ are not restricted to being non-negative; thus, random weights taking positive and negative values are permitted.

The main limitation of Theorem~\ref{theorem:ulln2} is that whenever $p>1$, the condition that $\sum_{j=1}^n \mathbb{E}(|W_{n j}|) = O(n^{1/p})$ requires the weights to be getting smaller as $n$ increases.  One can scale the weights up by a factor of $n^{1-1/p}$, but then the leading factor of $1/n^{1/p}$ in Equation~\ref{equation:ulln2} becomes $1/n$, effectively reverting back to the standard rate of convergence.
In Theorem~\ref{theorem:ulln3}, we show that this condition can be dropped if we assume stronger regularity and moment conditions.

\begin{condition}
Assume $\theta\mapsto H(\theta,x)$ is uniformly locally H\"{o}lder continuous, in the sense that there exist $a>0$, $M>0$, and $\delta>0$ such that for all $x\in\mathbb{R}^u$, $\theta,\theta'\in\Theta$, if $d(\theta,\theta')<\delta$ then $|H(\theta,x)-H(\theta',x)|\leq M d(\theta,\theta')^a$.
\label{condition:holder}
\end{condition}

\begin{condition}
Assume $N(r) \leq c\, r^{-D}$ for some $c>0$ and $D>0$, where $N(r)$ is the smallest number of open balls of radius $r$, centered at points in $\Theta$, needed to cover $\Theta$.
\label{condition:covering-number}
\end{condition}

Condition~\ref{condition:covering-number} holds for any compact $\Theta\subset\mathbb{R}^D$; indeed, by \citet[Example 27.1]{shalev2014understanding}, $N_1(r) \leq c_1 r^{-D}$ where $N_1(r)$
is the number of $r$-balls needed to cover $\Theta$, centered at any points in $\mathbb{R}^D$, and it is straightforward to verify that $N(r) \leq N_1(r/2)$.

\begin{theorem}\label{theorem:ulln3}
Let $X_1,X_2,\ldots \in \mathbb{R}^u$ i.i.d.\ and for $n\in\mathbb{N}$ independently, let $W_{n 1},\ldots,W_{n n}\in\mathbb{R}$ be NOD random variables, independent of $(X_1,X_2,\ldots)$.
Assume Conditions~\ref{condition:function-class}, \ref{condition:holder}, and \ref{condition:covering-number}.
Let $p\in(1,2)$ and suppose $\sup_{\theta\in\Theta}\mathbb{E}(|H(\theta,X_1)|^{q}) < \infty$, $\sum_{j=1}^n \mathbb{E}(|W_{n j}|^q) = O(n)$ for some $q > ((1-1/p)D/a + 1)/(1/p-1/2)$.  Then
\begin{align} 
\sup_{\theta \in \Theta} \bigg|\frac{1}{n^{1/p}} \sum_{j=1}^{n} \Big( W_{nj} H(\theta, X_j) - \mathbb{E}(W_{nj} H(\theta, X_j))\Big) \bigg| \xrightarrow[n\to\infty]{\mathrm{a.s.}} 0.  \label{equation:ulln3}
\end{align}
\end{theorem}

Note that as $p\to 1$ from above, the bound on the power $q$ approaches $2$.

\section{Examples of relevant resampling schemes}\label{sec:examples}

A key condition of Theorems~\ref{theorem:ulln2} and \ref{theorem:ulln3} is that the weights $W_{n 1},\ldots,W_{n n}$ must be NOD. It turns out that most popular resampling techniques satisfy this condition.

For instance, the $m$-out-of-$n$ bootstrap corresponds to $\mathrm{Multinomial}\big(m_n, (1/n,\ldots,1/n))$ weights, unequal probability with replacement corresponds to $\mathrm{Multinomial}\big(m_n, (p_{n 1},\ldots,p_{n n}))$ weights \citep{antal2011direct}, the Bayesian bootstrap corresponds to Dirichlet weights, the delete-$d$ jackknife and resampling without replacement correspond to multivariate hypergeometric weights \citep{chatterjee2005generalized}, the weighted likelihood bootstrap is equivalent to using independent weights \citep{newton1994approximate}, and the reweighting scheme of \cite{patak2009generalized} employs multivariate Gaussian weights with non-positive correlations. All of these distributions satisfy the NOD requirement \citep{joag1983negative}.

The NOD requirement is also satisfied by less standard reweighting schemes such as the downweight-$d$ jackknife \citep{chatterjee2005generalized} and simple random sampling with over-replacement \citep{antal2011simple}. 
In the downweight-$d$ jackknife, $d$ indices $i_{1}, \ldots, i_{d}$ are selected uniformly at random to be downweighted such that $W_{n i_{1}}= \cdots = W_{n i_{d}} = d/n$, whereas the remaining $n-d$ indices are upweighted to $1 + d/n$. These weights can be viewed as a monotonic transformation of the multivariate hypergeometric weights corresponding to the delete-$d$ jackknife, and thus are NOD \citep{chen2019strong}.
For simple random sampling with over-replacement, the weights can be viewed as the conditional distribution of
a sequence of $n$ independent geometric random variables, given that their sum equals $n$.   Geometric random variables satisfy the conditions of Theorem 2.6 from \citet{joag1983negative} according to \citet[3.1]{efron1965increasing}, implying that the resulting weights are NOD.

A general family of applicable NOD reweighting schemes can also be derived from the Farlie-Gumbel-Morgenstern (FGM) $n$-copula \citep[page 108]{nelsen2006introduction}. For variables $u_1, \ldots, u_n \in [0,1]$, an FGM copula is characterized by $2^n - n -1$ parameters---one for each subset of $\left\{1, \ldots, n\right\}$ that contains least two elements. Let $\theta_{1,2}, \theta_{1,3}, \ldots, \theta_{1, 2, \ldots, n} \in [-1,1]$ denote these parameters. The density of the FGM copula is given by
\begin{align}
f(u) &= 1 + \sum_{k=2}^n \sum_{1 \leq j_1 < j_2 < \cdots < j_k \leq n} \theta_{j_1, j_2, \cdots ,j_k} (1-2 u_{j_1}) (1 - 2 u_{j_2}) \cdots (1 - 2 u_{j_k}).
\end{align}
where the $\theta$ values must adhere to the constraints 
 \begin{align}
\sum_{k=2}^n \sum_{1 \leq j_1 < j_2 < \cdots < j_k \leq n} \theta_{j_1, j_2, \cdots, j_k} \epsilon_{j_1} \epsilon_{j_2} \cdots \epsilon_{j_k} \geq -1
\end{align}
for all $(\epsilon_1, \ldots, \epsilon_{n}) \in \left\{-1, 1\right\}^n$ to ensure non-negativity of the density function. It follows that each $u_{j}$ is marginally uniformly distributed, and for each $x \in [0,1]^n$,
\begin{align}
\mathbb{P}\left(u_{1} \leq x_{1},  \ldots, u_{n} \leq x_{n} \right) &= \left(\prod_{j=1}^n x_i\right) \left(1 + \sum_{k=2}^n   \sum_{1 \leq j_1 < j_2 < \cdots < j_k \leq n} \theta_{j_1, j_2, \cdots, j_k}\prod_{\ell = 1}^k(1-x_{j_\ell}) \right),\\
\mathbb{P}\left(u_{1} > x_{1},  \ldots, u_{n} > x_{n} \right) &= \left(\prod_{j=1}^n (1-x_i)\right) \left(1 + \sum_{k=2}^n   \sum_{1 \leq j_1 < j_2 < \cdots < j_k \leq n} \theta_{j_1, j_2, \cdots, j_k} (-1)^k \prod_{\ell = 1}^k x_{j_\ell} \right).
\end{align}
Therefore, for any $x_1, \ldots, x_n \in [0,1]$,
\begin{align}
\prod_{i=1}^n\mathbb{P}\left(u_{i} \leq x_{i}\right) = \prod_{i=1}^n x_i, \text{  and  } \prod_{i=1}^n\mathbb{P}\left(u_{i} > x_{i}\right) = \prod_{i=1}^n (1-x_i).
\end{align}
A FGM $n$-copula is thus NOD if and only if
\begin{align}
 \sum_{k=2}^n   \sum_{1 \leq j_1 < j_2 < \cdots < j_k \leq n} \theta_{j_1 j_2 \cdots j_k} \prod_{\ell = 1}^k(1-x_{j_\ell}) \leq 0, \text{ and} \label{criterion1}\\
 \sum_{k=2}^n   \sum_{1 \leq j_1 < j_2 < \cdots < j_k \leq n} \theta_{j_1 j_2 \cdots j_k} (-1)^k \prod_{\ell = 1}^k x_{j_\ell} \leq 0 \label{criterion2}
\end{align}
for all $x_1, \ldots, x_n \in [0,1]$. Because NOD is preserved under monotonic transformations, it follows that any multivariate distribution corresponding to such a FGM copula with $\theta$ satisfying the criteria in \ref{criterion1} and \ref{criterion2} will be NOD.

Finally, the condition that $\sum_{j=1}^n \mathbb{E}(|W_{n j}|^q) = O(n)$ in Theorem~\ref{theorem:ulln3} holds for all of the aforementioned reweighting procedures without any additional assumptions, except for a few cases. For the $\mathrm{Multinomial}\big(m_n, (p_{n 1},\ldots,p_{n n}))$ case, it holds as long as there exists a constant $\kappa > 0$ such that $p_{n j} < \kappa/n$ for all $j$ and $n$. For the Gaussian reweighting scheme \citep{patak2009generalized}, it holds as long as $\sup_{n,j} \mathrm{Var}(W_{nj})  < \infty$. For the independent weights and FGM copula cases, it holds when $\sup_{n,j} \mathbb{E}(|W_{n j}|^q) < \infty$.

\section{Proofs}
\label{section:proofs}

We begin by stating a strong law of large numbers of the bootstrap mean due to \citet[Theorem 2.1]{arenal1996unconditional}.
This lemma plays a key role in our proof of Theorem~\ref{theorem:ulln1}.

\begin{lemma}\label{lemma:bootstrap-mean}
Let $Z_1,Z_2,\ldots \in\mathbb{R}$ i.i.d.\ and for $n\in\mathbb{N}$ independently, let
\begin{align*}
(W_{n 1},\ldots,W_{n n}) \sim \mathrm{Multinomial}\big(m_n,  (1/n, \ldots, 1/n)\big)
\end{align*}
independently of $(Z_1,Z_2,\ldots)$, where $m_n$ is a positive integer for $n\in\mathbb{N}$. Suppose there exists a $\delta \in [0,1)$ such that
\begin{align}
\lim_{n \rightarrow \infty} \frac{n^{1 - \delta} \log(n)}{m_n} = 0 \text{~~ and ~~}\mathbb{E}\big( |Z_1|^{\frac{1}{1 - \delta}}\big) < \infty.
\end{align}
Then
\begin{align}
 \frac{1}{m_n} \sum_{j=1}^{n} W_{n j} Z_j \xrightarrow[n\to\infty]{\mathrm{a.s.}} \mathbb{E}(Z_1). 
\end{align}
\end{lemma}
\begin{proof}
See \citet[Theorem 2.1]{arenal1996unconditional} for the proof.
\end{proof}

\begin{proof}[{\bf Proof of Theorem~\ref{theorem:ulln1}}]
Our argument is based on the proof of Theorem 1.3.3 of \cite{GhoshJ.K2003BN}, except that we use Lemma~\ref{lemma:bootstrap-mean} in place of the strong law of large numbers for i.i.d.\ random variables.

Condition~\ref{condition:function-class} ensures that $x \mapsto \sup_{\theta \in \Theta} |H(\theta,x)|$ is measurable; this can be seen by letting $\theta_1,\theta_2,\ldots$ be a countable dense subset of $\Theta$, verifying that  $\sup_{j \in \mathbb{N}} |H(\theta_j,x)| = \sup_{\theta\in\Theta} |H(\theta,x)|$, and using \citet[Proposition 2.7]{folland2013real}.
Define $\mu(\theta) := \mathbb{E}\left(H(\theta, X_1) \right)$, and note that $\mu(\theta)$ is continuous by the dominated convergence theorem. 
Let $B_r(\theta_0) := \{\theta\in\Theta : d(\theta,\theta_0) < r\}$ denote the open ball of radius $r$ at $\theta_0$, where $d(\cdot,\cdot)$ is the metric on $\Theta$.
For $\theta\in\Theta$, $x\in\mathbb{R}^u$, and $r>0$, define
\begin{align}
\label{equation:eta}
\eta(\theta, x, r) := \sup_{\theta'\in B_r(\theta)} \Big\vert \big(H(\theta, x) - \mu(\theta) \big) - \big(H(\theta', x) - \mu(\theta') \big) \Big\vert,
\end{align}
and observe that by continuity and compactness,
\begin{align}
\eta(\theta,x,r) \leq 2 \sup_{\theta \in \Theta} |H(\theta, x)| + 2 \sup_{\theta \in \Theta}|\mu(\theta)| < \infty.
\end{align}
Applying the dominated convergence theorem again, we have 
$\lim_{r\to 0} \mathbb{E}(\eta(\theta,X_1,r)) = 0$ for all $\theta\in\Theta$.
Thus, for any $\varepsilon > 0$, by compactness of $\Theta$ there exist $K\in\mathbb{N}$, $\theta_1,\ldots,\theta_K\in\Theta$, and $r_1,\ldots,r_K > 0$ such that $\Theta = \textstyle{\bigcup_{i=1}^K} B_{r_i}(\theta_i)$ and $\mathbb{E}\left(\eta(\theta_i, X_1, r_i) \right) < \varepsilon$
for all $i \in \left\{1, \ldots, K \right\}$.  
Choosing $\delta \in (0, 1]$ according to the statement of Theorem~\ref{theorem:ulln1},
\begin{align}
\label{equation:eta-bound}
   \mathbb{E}\big(\eta(\theta_i, X_1, r_i)^{\frac{1}{1-\delta}}\big) & \leq  \mathbb{E}\Big(\big(2 \sup_{\theta\in\Theta} |H(\theta, X_1)| + 2 \sup_{\theta\in\Theta} |\mu(\theta)|\big)^{\frac{1}{1-\delta}}\Big)\\
   &\leq 4^{\frac{1}{1-\delta}} \mathbb{E}\Big(\sup_{\theta\in\Theta} |H(\theta, X_1)|^{\frac{1}{1-\delta}}  \Big) + 4^{\frac{1}{1-\delta}} \sup_{\theta\in\Theta} | \mu(\theta)|^{\frac{1}{1-\delta}} < \infty
\end{align}
since $(x + y)^{\frac{1}{1-\delta}} \leq (2\max\{|x|,|y|\})^{\frac{1}{1-\delta}} \leq 2^{\frac{1}{1-\delta}}(| x|^{\frac{1}{1-\delta}} + |y|^{\frac{1}{1-\delta}})$.
For all $i\in\{1,\ldots,K\}$, by applying Lemma~\ref{lemma:bootstrap-mean} with $Z_j = \eta(\theta_i,X_j,r_i)$ we have that
\begin{align}
\label{equation:eta-limit}
\frac{1}{m_n} \sum_{j=1}^{n} W_{n j}  \eta(\theta_i, X_j, r_i) \xrightarrow[n\to\infty]{\mathrm{a.s.}} \mathbb{E}\big(\eta(\theta_i, X_1, r_i)\big) < \varepsilon.
\end{align}
Similarly, by Lemma~\ref{lemma:bootstrap-mean} with $Z_j = H(\theta_i,X_j)$,
\begin{align}
\label{equation:H-limit}
\frac{1}{m_n} \sum_{j=1}^{n} W_{n j} H(\theta_i, X_j) \xrightarrow[n\to\infty]{\mathrm{a.s.}} \mu(\theta_i).
\end{align}
Thus, for any $\theta\in\Theta$, by choosing $i$ such that $\theta\in B_{r_i}(\theta_i)$, we have
\begin{align}
\label{equation:H-triangle}
&\bigg\vert\frac{1}{m_n} \sum_{j=1}^n W_{nj} H(\theta, X_j) - \mu(\theta) \bigg\vert\\
&~~~ \leq  \frac{1}{m_n} \sum_{j=1}^{n} W_{nj} \eta(\theta_i, X_j, r_i) + \bigg\vert\frac{1}{m_n} \sum_{j=1}^{n} W_{nj} H(\theta_i ,X_j) - \mu(\theta_i) \bigg\vert  \notag
\end{align}
by the triangle inequality and Equation~\ref{equation:eta}.
Letting $V_{n i}$ denote the right-hand side of Equation~\ref{equation:H-triangle}, we have $V_{n i}\to \mathbb{E}\big(\eta(\theta_i, X_1, r_i)\big) < \varepsilon$ a.s.\ by Equations~\ref{equation:eta-limit} and \ref{equation:H-limit}. Therefore,
\begin{align}
\label{equation:sup}
\sup_{\theta\in\Theta} \bigg\vert\frac{1}{m_n} \sum_{j=1}^n W_{n j} H(\theta, X_j) - \mu(\theta) \bigg\vert 
\leq \max_{1\leq i \leq K} V_{n i} 
\xrightarrow[n\to\infty]{\mathrm{a.s.}}
 \max_{1\leq i \leq K} \mathbb{E}\big(\eta(\theta_i, X_1, r_i)\big) < \varepsilon.
\end{align}
Since $\varepsilon>0$ is arbitrary, Equation~\ref{equation:sup} holds almost surely for $\varepsilon = \varepsilon_k = 1/k$ for all $k\in\mathbb{N}$, completing the proof.
\end{proof}


The following result, due to \cite{chen2019strong}, is a more general version of Lemma~\ref{lemma:bootstrap-mean} that extends beyond the standard bootstrap to negatively orthant dependent (NOD) weights.

\begin{lemma}\label{lemma:general-mean}
Let $Z_1,Z_2,\ldots \in \mathbb{R}$ be identically distributed NOD random variables.
For each $n\in\mathbb{N}$ independently, let $W_{n 1},\ldots,W_{n n}$ be NOD random variables, independent of $(Z_1,Z_2,\ldots)$. Suppose $\mathbb{E}(|Z_1|^{\beta}) < \infty$ and $\sum_{j=1}^n \mathbb{E}(|W_{n j}|^\alpha) = O(n)$ where $p \in [1,2)$ and either (a) $\alpha > 2p$ and $\beta > 1$ satisfy $\alpha^{-1} + \beta^{-1} = p^{-1}$, or (b) the weights $W_{n j}$ are identically distributed for all $n,j$, and $\alpha = \beta = 2 p$. Then
\begin{align} 
\frac{1}{n^{1/p}} \sum_{j=1}^{n} \big( W_{n j} Z_j - \mathbb{E}(W_{n j} Z_j) \big) \xrightarrow[n\to\infty]{\mathrm{a.s.}} 0. 
\end{align}
\end{lemma}
\begin{proof}
See \citet[Theorem 1 and Corollary 1]{chen2019strong} for the proof.
\end{proof}

\begin{proof}[{\bf Proof of Theorem~\ref{theorem:ulln2}}]
The proof is similar to the proof of Theorem~\ref{theorem:ulln1}, except that we use Lemma~\ref{lemma:general-mean} instead of Lemma~\ref{lemma:bootstrap-mean}, and some modifications are needed to handle more general distributions of the weights $W_{n j}$.

As in the proof of Theorem~\ref{theorem:ulln1}, 
define $\eta(\theta,x,r)$ by Equation~\ref{equation:eta}
and $\mu(\theta) := \mathbb{E}\left(H(\theta, X_1) \right)$.
As before, for any $\varepsilon > 0$, there exist $K\in\mathbb{N}$, $\theta_1,\ldots,\theta_K\in\Theta$, and $r_1,\ldots,r_K > 0$ such that $\Theta = \textstyle{\bigcup_{i=1}^K} B_{r_i}(\theta_i)$ and $\mathbb{E}\left(\eta(\theta_i, X_1, r_i) \right) < \varepsilon$
for all $i \in \left\{1, \ldots, K \right\}$. 
Just as in Equation~\ref{equation:eta-bound},
\begin{align*}
   \mathbb{E}\big(\eta(\theta_i, X_1, r_i)^{\beta}\big) 
   \leq 4^{\beta} \mathbb{E}\Big(\sup_{\theta\in\Theta} |H(\theta, X_1)|^{\beta}  \Big) + 4^{\beta} \sup_{\theta\in\Theta} | \mu(\theta)|^{\beta} < \infty.
\end{align*}
For all $i\in\{1,\ldots,K\}$, by applying Lemma~\ref{lemma:general-mean} with $Z_j = 1$ and $Z_j = H(\theta_i,X_j)$, respectively,
\begin{align}
& \frac{1}{n^{1/p}} \sum_{j=1}^{n} \big(W_{n j}  - \mathbb{E}(W_{nj}) \big) \xrightarrow[n\to\infty]{\mathrm{a.s.}}  0, 
\label{equation:1-limit2} \\
& \frac{1}{n^{1/p}} \sum_{j=1}^{n} \big(W_{n j}  H(\theta_i, X_j) - \mathbb{E}\left(W_{nj} \right)\mu(\theta_i) \big) \xrightarrow[n\to\infty]{\mathrm{a.s.}}  0.
\label{equation:H-limit2}
\end{align}
Let $W_{n j}^{+} = \max\{W_{n j},0\}$ and $W_{n j}^{-} = \max\{-W_{n j},0\}$.  Then $(W_{n 1}^{+},\ldots,W_{n n}^{+})$ and $(W_{n 1}^{-},\ldots,W_{n n}^{-})$ are each NOD because they are monotone transformations of the $W_{nj}$'s \citep{chen2019strong}.  Thus, Lemma~\ref{lemma:general-mean} applied to $Z_j = \eta(\theta_i,X_j,r_i)$ with $W_{nj}^{+}$ and $W_{nj}^{-}$, respectively, yields
\begin{align}
\frac{1}{n^{1/p}} \sum_{j=1}^{n}\Big( W^+_{n j}  \eta(\theta_i, X_j, r_i) - \mathbb{E}(W_{nj}^{+})\mathbb{E}\big(\eta(\theta_i, X_1, r_i)\big) \Big) & \xrightarrow[n\to\infty]{\mathrm{a.s.}}  0, \label{equation:eta-limit2pos}\\
\frac{1}{n^{1/p}} \sum_{j=1}^{n}\Big( W^-_{n j}  \eta(\theta_i, X_j, r_i) - \mathbb{E}(W_{nj}^{-})\mathbb{E}\big(\eta(\theta_i, X_1, r_i)\big) \Big) & \xrightarrow[n\to\infty]{\mathrm{a.s.}}  0 \label{equation:eta-limit2neg}.
\end{align}
By adding Equations~\ref{equation:eta-limit2pos} and \ref{equation:eta-limit2neg} and using the fact that $|W_{n j}| = W_{n j}^{+} + W_{n j}^{-}$, we have
\begin{align}
\frac{1}{n^{1/p}} \sum_{j=1}^{n}\Big(|W_{n j}|  \eta(\theta_i, X_j, r_i) - \mathbb{E}\left(|W_{nj}| \right)\mathbb{E}\big(\eta(\theta_i, X_1, r_i)\big) \Big) & \xrightarrow[n\to\infty]{\mathrm{a.s.}}  0, \label{equation:eta-limit2abs}
\end{align}
Since $\mathbb{E}(\eta(\theta_i, X_1, r_i)) < \varepsilon$, Equation~\ref{equation:eta-limit2abs} implies that, almost surely,
\begin{align}
\label{equation:eta-limsup}
\limsup_{n\to\infty}\frac{1}{n^{1/p}} \sum_{j=1}^{n} |W_{n j}| \eta(\theta_i, X_j, r_i) \leq
\limsup_{n\to\infty}\frac{1}{n^{1/p}} \sum_{j=1}^{n} \mathbb{E}(|W_{nj}| )\mathbb{E}\big(\eta(\theta_i, X_1, r_i)\big)
\leq C\varepsilon
\end{align}
where $C := \limsup_{n\to\infty} n^{-1/p} \sum_{j=1}^n \mathbb{E}(|W_{n j}|)$.
Note that $C<\infty$ by the assumption that $\sum_{j=1}^n \mathbb{E}(|W_{n j}|) = O(n^{1/p})$.
For any $\theta\in\Theta$, choosing $i$ such that $\theta\in B_{r_i}(\theta_i)$, we can write
\begin{align}
\label{equation:expand}
W_{n j} H(\theta,X_j) - \mathbb{E}(W_{n j})\mu(\theta)
&= W_{n j}\Big((H(\theta,X_j) - \mu(\theta)) - (H(\theta_i,X_j) - \mu(\theta_i))\Big) \notag\\
&~~~ + (W_{n j} - \mathbb{E}(W_{n j}))  (\mu(\theta) - \mu(\theta_i)) \\
&~~~ + \big(W_{n j} H(\theta_i,X_j) - \mathbb{E}(W_{n j})\mu(\theta_i)\big). \notag
\end{align}

Summing Equation~\ref{equation:expand} over $j$, using the triangle inequality,
and employing Equation~\ref{equation:eta}, 
\begin{align}
\label{equation:tri1}
 \bigg|\frac{1}{n^{1/p}}\sum_{j=1}^{n} \Big( W_{nj} H(\theta, X_j) - \mathbb{E}(W_{nj}) \mu(\theta) \Big) \bigg| 
&\leq \frac{1}{n^{1/p}} \sum_{j=1}^{n}  |W_{nj}| \,\eta(\theta_i,X_j,r_i) \\
& + \Big\vert \frac{1}{n^{1/p}}\sum_{j=1}^{n} \left( W_{nj} - \mathbb{E}(W_{nj})\right) \Big\vert \,\sup_{\theta\in\Theta} |\mu(\theta) - \mu(\theta_i)|\notag\\
& + \Big\vert\frac{1}{n^{1/p}}\sum_{j=1}^{n} \Big(W_{nj} H(\theta_i, X_j) - \mathbb{E}(W_{nj})\mu(\theta_i)\Big) \Big\vert.\notag
\end{align}
Letting $V_{n i}$ denote the right-hand side of Equation~\ref{equation:tri1}, we have $\limsup_{n} V_{n i}\leq C\varepsilon$ a.s.\ by Equations~\ref{equation:eta-limsup}, \ref{equation:1-limit2}, and
\ref{equation:H-limit2}, along with the fact that $\mu(\theta)$ is continuous and $\Theta$ is compact.
Therefore, almost surely,
\begin{align}
\label{equation:sup2}
\limsup_{n\to\infty}\; \sup_{\theta\in\Theta} \bigg\vert\frac{1}{n^{1/p}} \sum_{j=1}^n \Big(W_{n j} H(\theta, X_j) - \mathbb{E}(W_{n j})\mu(\theta)\Big) \bigg\vert 
\leq \limsup_{n\to\infty} \max_{1\leq i \leq K} V_{n i} \leq C \varepsilon.
\end{align}
As before, since $\varepsilon>0$ is arbitrary, Equation~\ref{equation:sup2} holds almost surely for $\varepsilon = \varepsilon_k = 1/k$ for all $k\in\mathbb{N}$, completing the proof.
\end{proof}


\begin{proof}[{\bf Proof of Theorem~\ref{theorem:ulln3}}]
First, we may assume without loss of generality that $W_{n j}$ and $H(\theta,X_j)$ are nonnegative since we can write $W_{n j} = W_{n j}^{+} - W_{n j}^{-}$ and $H(\theta,X_j) = H(\theta,X_j)^{+} - H(\theta,X_j)^{-}$, and apply the result in the nonnegative case to each of $W_{n j}^{+} H(\theta,X_j)^{+}$, $W_{n j}^{+} H(\theta,X_j)^{-}$, $W_{n j}^{-} H(\theta,X_j)^{+}$, and $W_{n j}^{-} H(\theta,X_j)^{-}$ to obtain the result for $W_{n j} H(\theta,X_j)$ in the general case; see the proof of Theorem 1 from \citet{chen2019strong} for a similar argument.

Let $\varepsilon > 0$.
As before, define $\eta(\theta,x,r)$ by Equation~\ref{equation:eta}
and $\mu(\theta) := \mathbb{E}\left(H(\theta, X_1) \right)$.
For $n\in\mathbb{N}$, define $r_n = (\varepsilon / n^{(1-1/p)})^{1/a}$ and $K_n = N(r_n)$, where $N(r)$ is defined in Condition~\ref{condition:covering-number}. 
Let $\theta_{n 1},\ldots,\theta_{n K_n}\in\Theta$ be the centers of $K_n$ balls of radius $r_n$ that cover $\Theta$, that is, $\bigcup_{i=1}^{K_n} B_{r_n}(\theta_{n i})$.

Let $C_W = 1 + \sup_{n} \frac{1}{n}\sum_{j=1}^n \mathbb{E}(W_{n j}^q) < \infty$.  Note that $q > 2 p$ since $p > 1$.
Thus, by Lemma~\ref{lemma:general-mean} with $Z_j = 1$,
\begin{align}
\frac{1}{n}\sum_{j=1}^n W_{n j} 
\leq C_W + \frac{1}{n}\sum_{j=1}^n (W_{n j} - \mathbb{E}(W_{n j}))
\xrightarrow[n\to\infty]{\mathrm{a.s.}} C_W
\end{align}
For all $n$ sufficiently large that $r_n < \delta$, we have 
$\eta(\theta,x,r_n) \leq 2 M r_n^a = 2 M \varepsilon / n^{(1-1/p)}$ by Condition~\ref{condition:holder}, and thus, 
\begin{align}
\label{equation:eta3}
\limsup_{n\to\infty}\max_{i\in[K_n]}\frac{1}{n^{1/p}} \sum_{j=1}^n W_{n j} \eta(\theta_{n i},X_j,r_n) 
\leq 2 M \varepsilon \limsup_{n\to\infty} \frac{1}{n} \sum_{j=1}^n W_{n j} \stackrel{\mathrm{a.s.}}{\leq} 2 M C_W \varepsilon 
\end{align}
where $[K_n] = \{1,\ldots,K_n\}$. By another application of Lemma~\ref{lemma:general-mean} with $Z_j = 1$,
\begin{align}
\label{equation:mu3}
\limsup_{n\to\infty}\max_{i\in[K_n]} \Big\vert \frac{1}{n^{1/p}} \sum_{j=1}^n (W_{n j} - \mathbb{E}(W_{n j}))\Big\vert \sup_{\theta\in\Theta}|\mu(\theta)-\mu(\theta_{n i})| 
\stackrel{\mathrm{a.s.}}{=} 0
\end{align}
since $\max_{i\in[K_n]}\sup_{\theta} |\mu(\theta)-\mu(\theta_{n i})| \leq 2 \sup_{\theta}|\mu(\theta)| < \infty$.
Defining $Y_{n i j} := W_{nj} H(\theta_{n i}, X_j) - \mathbb{E}(W_{n j})\mu(\theta_{n i})$, we claim that 
\begin{align}
\label{equation:Y}
\limsup_{n\to\infty}\max_{i\in[K_n]} \Big\vert\frac{1}{n^{1/p}}\sum_{j=1}^{n} Y_{n i j}\Big\vert \stackrel{\mathrm{a.s.}}{\leq} \varepsilon.
\end{align}
Assuming Equation~\ref{equation:Y} for the moment, we use the same decomposition as in Equation~\ref{equation:tri1}
and plug in Equations~\ref{equation:eta3}--\ref{equation:Y} to obtain
\begin{align*}
\limsup_{n\to\infty} \sup_{\theta\in\Theta}\bigg\vert\frac{1}{n^{1/p}}\sum_{j=1}^{n} & \Big( W_{n j} H(\theta, X_j) - \mathbb{E}(W_{n j}) \mu(\theta) \Big) \bigg\vert
\leq 2 M C_W \varepsilon + 0 + \varepsilon.
\end{align*}
Since $\varepsilon > 0$ is arbitrarily small, this will yield the result of the theorem.

To complete the proof, we need to show Equation~\ref{equation:Y}.
For each $n$ and $i$, by \citet[Lemma 1]{chen2019strong}, $Y_{n i 1},\ldots,Y_{n i n}$ is an NOD sequence since we have assumed without loss of generality that $W_{n j}$ and $H(\theta,X_j)$ are nonnegative, and adding constants preserves the NOD property. \cite{asadian2006rosenthal} and \cite{rivaz2007moment} provide moment inequalities that are useful in this context.
By \citet[Corollary 2.2]{asadian2006rosenthal} (also see \cite[Corollary 3]{rivaz2007moment}), since $\mathbb{E}(Y_{n i j}) = 0$ and $q > 2$, 
\begin{align}
\label{equation:asadian-bound}
\mathbb{E}\Big(\Big\vert\sum_{j=1}^n Y_{n i j}\Big\vert^q\Big)
\leq C_A(q)\sum_{j=1}^n \mathbb{E}(|Y_{n i j}|^q) + C_A(q) \Big(\sum_{j=1}^n \mathbb{E}(|Y_{n i j}|^2)\Big)^{q/2}
\end{align}
where $C_A(q)$ is a universal constant that depends only on $q$.
Letting $C_H = 1 + \sup_{\theta\in\Theta} \mathbb{E}(|H(\theta,X_1)|^q) < \infty$, we have
\begin{align}
\label{equation:Yq}
\sum_{j=1}^n \mathbb{E}(|Y_{n i j}|^q) \leq 
\sum_{j=1}^n 2^{q+1} \mathbb{E}\big(|W_{n j}|^q |H(\theta_{n i},X_j)|^q\big)
\leq 2^{q+1} C_H C_W n.
\end{align}
Likewise, since $\mathrm{Var}(X) \leq \mathbb{E}(X^2)$ for any random variable $X$,
\begin{align}
\label{equation:Y2}
\sum_{j=1}^n \mathbb{E}(|Y_{n i j}|^2) \leq 
\sum_{j=1}^n \mathbb{E}\big(|W_{n j}|^2 |H(\theta_{n i},X_j)|^2\big)
\leq C_H C_W n.
\end{align}
Plugging Equations~\ref{equation:Yq} and \ref{equation:Y2} into Equation~\ref{equation:asadian-bound}, we have
\begin{align}
\label{equation:Sq}
\mathbb{E}\Big(\Big\vert\sum_{j=1}^n Y_{n i j}\Big\vert^q\Big)
\leq C n^{q/2}
\end{align}
where $C = C_A(q)2^{q+1} C_H C_W + C_A(q) (C_H C_W)^{q/2}$.
By Condition~\ref{condition:covering-number}, $K_n = N(r_n) \leq c\, r_n^{-D} = c\, n^{(1-1/p)D/a} / \varepsilon^{D/a}$.
Along with Markov's inequality and Equation~\ref{equation:Sq}, this implies
\begin{align*}
\mathbb{P}\Big(\max_{i\in[K_n]} \Big\vert \frac{1}{n^{1/p}}\sum_{j=1}^n Y_{n i j}\Big\vert > \varepsilon\Big) 
&\leq \sum_{i=1}^{K_n} \mathbb{P}\Big(\Big\vert \frac{1}{n^{1/p}}\sum_{j=1}^n Y_{n i j}\Big\vert > \varepsilon\Big) 
\leq \frac{1}{\varepsilon^q n^{q/p}} \sum_{i=1}^{K_n} \mathbb{E}\Big(\Big\vert \sum_{j=1}^n Y_{n i j}\Big\vert^q\Big) \\
&\leq \frac{K_n C n^{q/2}}{\varepsilon^q n^{q/p}}
\leq \frac{c\, C}{\varepsilon^q \varepsilon^{D/a}} n^{(1-1/p)D/a} \,n^{q(1/2-1/p)}
= \frac{c\, C}{\varepsilon^q \varepsilon^{D/a}} n^{-\gamma}
\end{align*}
where $\gamma := q(1/p-1/2) - (1-1/p)D/a > 1$ because $q > ((1-1/p)D/a + 1)/(1/p-1/2)$ by assumption. Hence, 
\begin{align}
\sum_{n=1}^\infty \mathbb{P}\Big(\max_{i\in[K_n]} \Big\vert \frac{1}{n^{1/p}}\sum_{j=1}^n Y_{n i j}\Big\vert > \varepsilon\Big)
\leq \frac{c\, C}{\varepsilon^q \varepsilon^{D/a}} \sum_{n=1}^\infty n^{-\gamma} < \infty.
\end{align}
Therefore, Equation~\ref{equation:Y} holds by the Borel--Cantelli lemma.  This completes the proof.
\end{proof}

\bibliography{refs.bib}

\begin{thebibliography}{31}
\providecommand{\natexlab}[1]{#1}
\providecommand{\url}[1]{\texttt{#1}}
\expandafter\ifx\csname urlstyle\endcsname\relax
  \providecommand{\doi}[1]{doi: #1}\else
  \providecommand{\doi}{doi: \begingroup \urlstyle{rm}\Url}\fi

\bibitem[Antal and Till{\'e}(2011{\natexlab{a}})]{antal2011direct}
E.~Antal and Y.~Till{\'e}.
\newblock A direct bootstrap method for complex sampling designs from a finite
  population.
\newblock \emph{Journal of the American Statistical Association}, 106\penalty0
  (494):\penalty0 534--543, 2011{\natexlab{a}}.

\bibitem[Antal and Till{\'e}(2011{\natexlab{b}})]{antal2011simple}
E.~Antal and Y.~Till{\'e}.
\newblock Simple random sampling with over-replacement.
\newblock \emph{Journal of Statistical Planning and Inference}, 141\penalty0
  (1):\penalty0 597--601, 2011{\natexlab{b}}.

\bibitem[Arenal-Guti{\'e}rrez et~al.(1996)Arenal-Guti{\'e}rrez, Matr{\'a}n, and
  Cuesta-Albertos]{arenal1996unconditional}
E.~Arenal-Guti{\'e}rrez, C.~Matr{\'a}n, and J.~A. Cuesta-Albertos.
\newblock On the unconditional strong law of large numbers for the bootstrap
  mean.
\newblock \emph{Statistics \& Probability Letters}, 27\penalty0 (1):\penalty0
  49--60, 1996.

\bibitem[Asadian et~al.(2006)Asadian, Fakoor, and
  Bozorgnia]{asadian2006rosenthal}
N.~Asadian, V.~Fakoor, and A.~Bozorgnia.
\newblock Rosenthal's Type Inequalities for Negatively Orthant Dependent Random
  Variables.
\newblock \emph{Journal of the Iranian Statistical Society}, 5\penalty0 (1 and
  2):\penalty0 69--75, 2006.

\bibitem[Athreya et~al.(1984)Athreya, Ghosh, Low, and Sen]{athreya1984laws}
K.~B. Athreya, M.~Ghosh, L.~Y. Low, and P.~K. Sen.
\newblock Laws of large numbers for bootstrapped U-statistics.
\newblock \emph{Journal of Statistical Planning and Inference}, 9\penalty0
  (2):\penalty0 185--194, 1984.

\bibitem[Bickel et~al.(2012)Bickel, G{\"o}tze, and van
  Zwet]{bickel2012resampling}
P.~J. Bickel, F.~G{\"o}tze, and W.~R. van Zwet.
\newblock Resampling fewer than n observations: gains, losses, and remedies for
  losses.
\newblock In \emph{Selected works of Willem van Zwet}, pages 267--297.
  Springer, 2012.

\bibitem[Breiman(1996)]{breiman1996bagging}
L.~Breiman.
\newblock Bagging predictors.
\newblock \emph{Machine Learning}, 24\penalty0 (2):\penalty0 123--140, 1996.

\bibitem[Chatterjee and Bose(2005)]{chatterjee2005generalized}
S.~Chatterjee and A.~Bose.
\newblock Generalized bootstrap for estimating equations.
\newblock \emph{The Annals of Statistics}, 33\penalty0 (1):\penalty0 414--436,
  2005.

\bibitem[Chen et~al.(2019)Chen, Zhang, and Sung]{chen2019strong}
P.~Chen, T.~Zhang, and S.~H. Sung.
\newblock Strong laws for randomly weighted sums of random variables and
  applications in the bootstrap and random design regression.
\newblock \emph{Statistica Sinica}, 29\penalty0 (4):\penalty0 1739--1749, 2019.

\bibitem[Efron(1965)]{efron1965increasing}
B.~Efron.
\newblock Increasing properties of {P\'{o}lya} frequency function.
\newblock \emph{The Annals of Mathematical Statistics}, pages 272--279, 1965.

\bibitem[Efron(1979)]{efron1979bootstrap}
B.~Efron.
\newblock Bootstrap methods: Another look at the jackknife.
\newblock \emph{The Annals of Statistics}, pages 1--26, 1979.

\bibitem[Efron(1982)]{efron1982jackknife}
B.~Efron.
\newblock \emph{The jackknife, the bootstrap and other resampling plans}.
\newblock SIAM, 1982.

\bibitem[Efron and Tibshirani(1994)]{efron1994introduction}
B.~Efron and R.~J. Tibshirani.
\newblock \emph{An introduction to the bootstrap}.
\newblock CRC press, 1994.

\bibitem[Folland(2013)]{folland2013real}
G.~B. Folland.
\newblock \emph{{Real Analysis: Modern Techniques and Their Applications}}.
\newblock John Wiley \& Sons, 2013.

\bibitem[Ghosh and Ramamoorthi(2003)]{GhoshJ.K2003BN}
J.~K. Ghosh and R.~V. Ramamoorthi.
\newblock \emph{Bayesian Nonparametrics}.
\newblock Springer Series in Statistics. Springer, New York, NY, 2003.
\newblock ISBN 0387955372.

\bibitem[Gin{\'e} and Zinn(1990)]{gine1990bootstrapping}
E.~Gin{\'e} and J.~Zinn.
\newblock Bootstrapping general empirical measures.
\newblock \emph{The Annals of Probability}, pages 851--869, 1990.

\bibitem[Huggins and Miller(2019)]{huggins2019robust}
J.~H. Huggins and J.~W. Miller.
\newblock Robust inference and model criticism using bagged posteriors.
\newblock \emph{arXiv preprint arXiv:1912.07104}, 2019.

\bibitem[Huggins and Miller(2022)]{10.1214/21-BA1301}
J.~H. Huggins and J.~W. Miller.
\newblock {Reproducible model selection using bagged posteriors}.
\newblock \emph{Bayesian Analysis}, pages 1--26, 2022.

\bibitem[Joag-Dev and Proschan(1983)]{joag1983negative}
K.~Joag-Dev and F.~Proschan.
\newblock Negative association of random variables with applications.
\newblock \emph{The Annals of Statistics}, pages 286--295, 1983.

\bibitem[Kosorok(2008)]{kosorok2008introduction}
M.~R. Kosorok.
\newblock \emph{Introduction to Empirical Processes and Semiparametric
  Inference}.
\newblock Springer, 2008.

\bibitem[Lehmann(1966)]{lehmann1966some}
E.~L. Lehmann.
\newblock Some concepts of dependence.
\newblock \emph{The Annals of Mathematical Statistics}, 37\penalty0
  (5):\penalty0 1137--1153, 1966.

\bibitem[Little and Badawy(2019)]{little2019causal}
M.~A. Little and R.~Badawy.
\newblock Causal bootstrapping.
\newblock \emph{arXiv preprint arXiv:1910.09648}, 2019.

\bibitem[Nelsen(2006)]{nelsen2006introduction}
R.~B. Nelsen.
\newblock \emph{An Introduction to Copulas}.
\newblock Springer, 2006.

\bibitem[Newton and Raftery(1994)]{newton1994approximate}
M.~A. Newton and A.~E. Raftery.
\newblock Approximate {Bayesian} inference with the weighted likelihood
  bootstrap.
\newblock \emph{Journal of the Royal Statistical Society: Series B
  (Methodological)}, 56\penalty0 (1):\penalty0 3--26, 1994.

\bibitem[Patak and Beaumont(2009)]{patak2009generalized}
Z.~Patak and J.-F. Beaumont.
\newblock Generalized bootstrap for prices surveys.
\newblock \emph{57th Session of the International Statistical Institute,
  Durban, South-Africa}, 2009.

\bibitem[Rivaz et~al.(2007)Rivaz, Amini, and Bozorgnia]{rivaz2007moment}
F.~Rivaz, M.~Amini, and A.~G. Bozorgnia.
\newblock Moment Inequalities and Applications for Negative Dependence Random
  Variables.
\newblock 26\penalty0 (4):\penalty0 7--11, 2007.

\bibitem[Rubin(1981)]{rubin1981bayesian}
D.~B. Rubin.
\newblock The {Bayesian} bootstrap.
\newblock \emph{The Annals of Statistics}, pages 130--134, 1981.

\bibitem[Shalev-Shwartz and Ben-David(2014)]{shalev2014understanding}
S.~Shalev-Shwartz and S.~Ben-David.
\newblock \emph{Understanding machine learning: From theory to algorithms}.
\newblock Cambridge university press, 2014.

\bibitem[Singh(1981)]{singh1981asymptotic}
K.~Singh.
\newblock On the asymptotic accuracy of {Efron's} bootstrap.
\newblock \emph{The Annals of Statistics}, pages 1187--1195, 1981.

\bibitem[Singh and Xie(2003)]{singh2003bootlier}
K.~Singh and M.~Xie.
\newblock Bootlier-plot: bootstrap based outlier detection plot.
\newblock \emph{Sankhy{\=a}: The Indian Journal of Statistics}, pages 532--559,
  2003.

\bibitem[Vaart and Wellner(1996)]{vaart1996weak}
A.~W. Vaart and J.~A. Wellner.
\newblock \emph{Weak Convergence and Empirical Processes}.
\newblock Springer, 1996.

\end{thebibliography}
\bibliographystyle{abbrvnatcap}

\end{document}